\documentclass[12pt]{amsart}
\usepackage[utf8]{inputenc}
\usepackage{indentfirst}
\usepackage{graphicx}
\usepackage{amsmath}
\usepackage{MnSymbol}
\usepackage{hyperref}
\usepackage{float}  
\usepackage{asymptote}
\usepackage{tikz}
\usepackage{xcolor}
\usepackage{comment}
\usepackage{enumerate}
\usepackage{enumitem}

\DeclareMathOperator{\Path}{\mathsf{Path}}
\DeclareMathOperator{\Cycle}{\mathsf{Cycle}}
\newcommand{\Tadpole}{\mathsf{Tad}}
\newcommand{\Fruit}{\mathsf{Cycle}^{\perp}}

\DeclareMathOperator{\Star}{\mathsf{Star}}
\DeclareMathOperator{\Grid}{\mathsf{Grid}}
\DeclareMathOperator{\FS}{\mathsf{FS}}
\DeclareMathOperator{\Spider}{\mathsf{Spider}}
\DeclareMathOperator{\Spycle}{\mathsf{SC}}

\usepackage[margin=1in]{geometry}

\theoremstyle{plain}
\newtheorem{theorem}{Theorem}
\newtheorem{lemma}[theorem]{Lemma}
\newtheorem{corollary}[theorem]{Corollary}

\newtheorem{qn}[theorem]{Question}
\newtheorem{claim}[theorem]{Claim}

\theoremstyle{definition}

\theoremstyle{remark}

\newtheorem*{ex}{Example}

\numberwithin{equation}{section}
\numberwithin{theorem}{section}

\title{Connectedness in Friends-and-Strangers Graphs of Spiders and Complements}

\begin{document}

\author[Alan Lee]{Alan Lee}
\address[]{Henry M. Gunn High School, Palo Alto, CA 94306, USA}
\email{alandongjinlee@gmail.com}
\maketitle

\begin{abstract}
    Let $X$ and $Y$ be two graphs with vertex set $[n]$. Their friends-and-strangers graph $\FS(X,Y)$ is a graph with vertices corresponding to elements of the group $S_n$, and two permutations $\sigma$ and $\sigma'$ are adjacent if they are separated by a transposition $\{a,b\}$ such that $a$ and $b$ are adjacent in $X$ and $\sigma(a)$ and $\sigma(b)$ are adjacent in $Y$. Specific friends-and-strangers graphs such as $\FS(\Path_n,Y)$ and $\FS(\Cycle_n,Y)$ have been researched, and their connected components have been enumerated using various equivalence relations such as double-flip equivalence. A spider graph is a collection of path graphs that are all connected to a single center point. In this paper, we delve deeper into the question of when $\FS(X,Y)$ is connected when $X$ is a spider and $Y$ is the complement of a spider or a tadpole.
\end{abstract}



\section{Introduction}

Suppose that a set of $n$ people are sitting on $n$ chairs, playing a swapping game. Each pair of chairs can be adjacent or not, and a swap between two people sitting in adjacent chairs is only valid if the two people are friends (where friendship is symmetric). The set of all permutations of the $n$ people on the $n$ chairs produces a graph, where two permutations are adjacent if they can be obtained from each other by an allowable swap.

We acquaint ourselves with some basic graph theoretical terminology before formally introducing friends-and-strangers graphs. We write $V(G)$ and $E(G)$ for the vertex set and edge set, respectively, of a graph $G$. We define a \textit{subgraph} of a graph to be a graph whose vertex \textit{and} edge sets are subsets of the original graph. The \textit{complement} of a graph $G$, denoted $\overline{G}$, is the graph with vertex set $V(\overline{G}) = V(G)$ and edge set $E(\overline{G}) = \{\{a,b\}:\{a,b\}\not\in E(G)\}$. Additionally, a graph $G$ is \textit{connected} if there exists a sequence of adjacent vertices from $a$ to $b$ for all vertices $a,b\in G$.

We can 
formalize the swapping game notion as follows. Let $X$ and $Y$ be graphs with $V(X)=V(Y)=[n]=\{1,\ldots,n\}$. The \emph{friends-and-strangers graph} $\FS(X,Y)$ has a vertex set isomorphic to the elements of the permutation group $S_n$, and two vertices (permutations) $\sigma,\phi$ are connected if there exists an edge $\{a,b\}\in E(X)$ such that the following conditions are met:
\begin{itemize}
    \item $\{\sigma(a),\sigma(b)\}\in E(Y)$
    \item $\sigma(a)=\phi(b)$ and $\sigma(b)=\phi(a)$
    \item $\sigma(c)=\phi(c) \text{ for all }  c\in V(X)\backslash \{a,b\}.$
\end{itemize}

To avoid confusion when discussing the vertices of $X$ and $Y$, we will refer to the vertices of $Y$ as \emph{labels}.

We also define special graphs that will be frequented in the following sections. Note that all of the following graphs have vertex set $[n]$, where $n$ can be changed as necessary.

\begin{itemize}
    \item The \textit{star graph} $\Star_n$ is the graph with edge set $E(\Star_n)=\{\{1,b\}:b\in[n],1<b\}$.
    \item The \textit{path graph} $\Path_n$ is the graph with edge set $E(\Path_n)=\{\{k,k+1\}:k\in[n-1]\}$.
    \item The \textit{cycle graph} $\Cycle_n$ is the path graph with the additional edge $\{n,1\}$.
    \item The \textit{tadpole graph} $\Tadpole_{c,n-c}$ is a cycle graph with $c\geq 3$ edges joined together with a path graph with $n-c$ edges. Its vertex set is $[n]$ and its edge set is $\{\{k,k+1\}:k\in[n-1]\}\cup\{\{1,c\}\}.$ We also denote the sole vertex of degree $3$ as the \textit{triple point}. Additionally, letting $n=c$ results in $\Tadpole_{c,n-c}\cong \Cycle_{n}$.
\end{itemize}

\begin{ex}
Let $X=\Star_{4}=\begin{array}{l}
\begin{tikzpicture}[line width=1pt,x=0.4cm,y=0.3cm] \tikzstyle{dot}=[circle,thin,draw,fill=black,inner sep=1pt] 
\pgfdeclarelayer{nodelayer} \pgfdeclarelayer{edgelayer} \pgfsetlayers{edgelayer,nodelayer}
\begin{pgfonlayer}{nodelayer}
\node [style=dot] (0) at (0,0) {};
\node [style=dot] (1) at (2,0) {};
\node [style=dot] (2) at (-1,1) {};
\node [style=dot] (3) at (-1,-1) {};
\end{pgfonlayer}
\begin{pgfonlayer}{edgelayer}
		\draw (1) to (0);
		\draw (0) to (2);
		\draw (3) to (0);
	\end{pgfonlayer}\end{tikzpicture}\end{array}$ and
$Y=\Tadpole_{3,1}=\begin{array}{l}\begin{tikzpicture}[line width=1pt,x=0.4cm,y=0.3cm] \tikzstyle{dot}=[circle,thin,draw,fill=black,inner sep=1pt] 
\pgfdeclarelayer{nodelayer} \pgfdeclarelayer{edgelayer} \pgfsetlayers{edgelayer,nodelayer}
\begin{pgfonlayer}{nodelayer}
\node [style=dot] (0) at (0,0) {};
\node [style=dot] (1) at (2,0) {};
\node [style=dot] (2) at (-1.5,1) {};
\node [style=dot] (3) at (-1.5,-1) {};
\end{pgfonlayer}
\begin{pgfonlayer}{edgelayer}
		\draw (1) to (0);
		\draw (0) to (2);
		\draw (3) to (0);
		\draw (2) to (3);
	\end{pgfonlayer}\end{tikzpicture}\end{array}$. 
In Figure \ref{fig:examplefsgraph} the graph $\FS(X,Y)$ is displayed.
\end{ex}
\begin{figure}[H]
    \centering
    \begin{tikzpicture}[line width=1pt, x=0.55cm,y=0.3cm]
    \tikzstyle{dot}=[circle,thin,draw,fill=black,inner sep=1pt]
    \pgfdeclarelayer{nodelayer}
    \pgfdeclarelayer{edgelayer}
    \pgfsetlayers{edgelayer,nodelayer}
	\begin{pgfonlayer}{nodelayer}
		\node [style=dot] (0) at (-7, 2) {};
		\node [style=dot] (1) at (-7, 3) {};
		\node [style=dot] (2) at (-8, 1) {};
		\node [style=dot] (3) at (-6, 1) {};
		\node [style=dot] (4) at (-8, -1) {};
		\node [style=dot] (5) at (-6, -1) {};
		\node [style=dot] (6) at (-7, -2) {};
		\node [style=dot] (7) at (-7, -3) {};
		\node [style=dot] (8) at (-3, 2) {};
		\node [style=dot] (9) at (-3, 3) {};
		\node [style=dot] (10) at (-4, 1) {};
		\node [style=dot] (11) at (-2, 1) {};
		\node [style=dot] (12) at (-4, -1) {};
		\node [style=dot] (13) at (-2, -1) {};
		\node [style=dot] (14) at (-3, -2) {};
		\node [style=dot] (15) at (-3, -3) {};
		\node [style=dot] (16) at (1, 2) {};
		\node [style=dot] (17) at (1, 3) {};
		\node [style=dot] (18) at (0, 1) {};
		\node [style=dot] (19) at (2, 1) {};
		\node [style=dot] (20) at (0, -1) {};
		\node [style=dot] (21) at (2, -1) {};
		\node [style=dot] (22) at (1, -2) {};
		\node [style=dot] (23) at (1, -3) {};
	\end{pgfonlayer}
	\begin{pgfonlayer}{edgelayer}
		\draw (1) to (0);
		\draw (0) to (2);
		\draw (2) to (4);
		\draw (4) to (6);
		\draw (6) to (7);
		\draw (6) to (5);
		\draw (5) to (3);
		\draw (3) to (0);
		\draw (9) to (8);
		\draw (8) to (10);
		\draw (10) to (12);
		\draw (12) to (14);
		\draw (14) to (15);
		\draw (14) to (13);
		\draw (13) to (11);
		\draw (11) to (8);
		\draw (17) to (16);
		\draw (16) to (18);
		\draw (18) to (20);
		\draw (20) to (22);
		\draw (22) to (23);
		\draw (22) to (21);
		\draw (21) to (19);
		\draw (19) to (16);
	\end{pgfonlayer}
\end{tikzpicture}
    \label{fig:examplefsgraph}
    \caption{The graph $\FS(\Star_4,\Tadpole_{3,1})$ without vertex labels.}
\end{figure}

Friends-and-strangers graphs help explain how permutations are related. For example, the fact that the 15-puzzle is not solvable 
is due to the fact that $\FS(\Star_{16},\Grid_{4\times 4})$ is not connected, where $\Grid_{n\times n}$ is a square grid graph with $n^2$ vertices. This example was studied by Wilson in \cite{wilson}, where he studied the connectivity of friends-and-strangers graphs of the form $\FS(\Star_n,Y)$.

The term ``friends-and-strangers" was first coined by Defant and Kravitz in \cite{main}, where the connectivity of $\FS(\Cycle_n,Y)$ and $\FS(\Path_n,Y)$ was studied. Extremal aspects of friends-and-strangers graphs were studied by Alon, Defant, and Kravitz in \cite{typext} as well as Bangachev in \cite{bang}, where a minimal degree condition on $X$ and $Y$ was provided for the connectedness of $\FS(X,Y).$ Jeong also determined when $\FS(X,Y)$ is connected for biconnected graphs $X$ \cite{jeongstruct}. Additionally, Wang and Chen studied the connectivity of $\FS(X,Y)$ for randomly selected pairs of graphs $X\in\mathcal{G}(n,p_1)$ and $Y\in\mathcal{G}(n,p_2)$ \cite{randpairs}. Although the connectedness of friends-and-strangers graphs has been the primary question for related research, Jeong's recent paper shows that diameters of friends-and-strangers graphs are not polynomially bounded \cite{jeongdiam}. 

In this paper we expand upon the results shown in \cite{group}, specifically those relating to connectedness of $\FS(\Spider,Y)$. After introducing some preliminary results in \ref{background}, we expand upon them in Section \ref{fruit} to consider the connectedness of the graph $\FS(\Spider,\overline{\Tadpole_{c,n-c}})$. We then provide an inductive approach to show how every connected graph $X$ 
that contains $\Spider(2,2,1,1)$ produces a connected friends-and-strangers graph $\FS(X,\overline{\Tadpole}_{c,n-c})$. 

In Section \ref{general}, we focus on finding graphs $X$ such that whenever $\overline{Y}$ is a spider with at most $n$ legs, $\FS(X,Y)$ is connected. We prove necessary and sufficient conditions guaranteeing the connectedness of $\FS(X,Y)$. These conditions concern whether or not $X$ contains a specific graph.

In Section \ref{final}, we propose some more questions regarding the connectedness of friends-and-stranger graphs $\FS(X,Y)$ where $X$ is a modified type of spider called the \emph{spycle}, as well as if $Y$ is the complement of a spycle graph.

\section{Background} \label{background}

Let 
$\lambda_1\geq\lambda_2\geq\cdots\geq\lambda_k$ be integers. The \emph{spider graph} $\Spider(\lambda_1,\lambda_2,\ldots,\lambda_{k})$ is a graph on $n=1+\sum_{i=1}^k \lambda_i$ vertices. 
Each vertex of the form $1,1+\lambda_1,1+\lambda_1+\lambda_2,\ldots$ is adjacent to the vertex $n$. For all vertices $m\neq n$, the vertex $m$ is adjacent to vertex $m+1$ if and only if the vertex $m+1$ is not adjacent to vertex $n$. Each of the paths that have been adjoined to form the spider graph are referred to as \emph{legs}, and the vertex of degree $k$ is referred to as the \emph{center}. Additionally, the vertices of degree $1$ are referred to as \emph{feet}. A picture of $\Spider(5,3,1)$ is below.
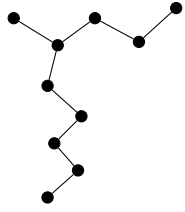
\begin{figure}[H]
\begin{tikzpicture}[line width=1pt, x=0.18cm,y=0.18cm]
\tikzstyle{dot}=[circle,thin,draw,fill=black,inner sep=1.5pt]
\pgfdeclarelayer{nodelayer}
\pgfdeclarelayer{edgelayer}
\pgfsetlayers{edgelayer,nodelayer}
	\begin{pgfonlayer}{nodelayer}
		\node [style=dot] (0) at (2.75, 6.25) {};
		\node [style=dot] (1) at (-0.5, 8.25) {};
		\node [style=dot] (2) at (5.5, 8.25) {};
		\node [style=dot] (3) at (8.75, 6.5) {};
		\node [style=dot] (4) at (11.5, 9) {};
		\node [style=dot] (5) at (2, 3.25) {};
		\node [style=dot] (6) at (4.5, 1) {};
		\node [style=dot] (7) at (2.5, -1) {};
		\node [style=dot] (8) at (4.25, -3) {};
		\node [style=dot] (9) at (2, -5) {};
	\end{pgfonlayer}
	\begin{pgfonlayer}{edgelayer}
		\draw (1) to (0);
		\draw (0) to (2);
		\draw (2) to (3);
		\draw (3) to (4);
		\draw (0) to (5);
		\draw (5) to (6);
		\draw (6) to (7);
		\draw (7) to (8);
		\draw (8) to (9);
	\end{pgfonlayer}
\end{tikzpicture}
\end{figure}

The idea of a \textit{fruit graph} $\Fruit_n$ was also introduced in \cite{group}, where $\Fruit_n=\Tadpole_{n-1,1}$. The following results were shown.

\begin{theorem}[\cite{group}]\label{thm:group}
Let $\lambda_1\geq\cdots\geq\lambda_k$ be positive integers such that $\lambda_1+\cdots+\lambda_k+1=n\geq 4$. The friends-and-strangers graph $\FS(\Spider(\lambda_1,\ldots,\lambda_k),\overline{\Cycle_{n}})$ is connected if and only if $(\lambda_1,\ldots,\lambda_k)$ is not of the form $(\lambda_1,1,1)$ and is not in the following list:
    \[ (1,1,1,1),\quad (2,2,1),\quad  (2,2,2),\quad  (3,2,1),\quad  (3,3,1),\quad  (4,2,1),\quad  (5,2,1).\]
\end{theorem}

In particular, we use the following consequence of Theorem \ref{thm:group} to extend results related to connectedness.

\begin{corollary}\label{prop3.5}
Let $X$ be a connected graph on $n\geq 6$ vertices with at least one vertex of degree $4$ or more. Then $\FS(X,\overline{\Cycle_{n}})$ is connected.
\end{corollary}

In \cite{group} the following results regarding the connectedness of $\FS(\Spider,\overline{\Fruit})$ were shown. In this paper, we expand upon this result by considering if the graph $Y$ is the complement of a tadpole graph, which yields a general result that encompasses when $Y$ is the complement of a $3$-legged spider.

\begin{theorem}[\cite{group}]
Let $\lambda_1\geq\cdots\geq\lambda_k$ be positive integers such that $k\geq 3$ and $\lambda_1+\cdots+\lambda_k+1=n$. Then $\FS(\Spider(\lambda_1,\ldots,\lambda_k),\overline{\Fruit_n})$ is disconnected if and only if  $(\lambda_1,\ldots,\lambda_k)$ is of one of the following forms: \[(\lambda_1,1,1,1),\quad(\lambda_1,\lambda_2,1),\quad (2,2,2).\]
\end{theorem}

\begin{corollary}
Let $a$ and $b$ be positive integers with $a\geq b$. Then $\FS(\Spider(a,b,1,1),\overline{\Fruit_{n}})$ is connected if and only if $b\geq 2$.
\end{corollary}

\section{Fruits to Tadpoles}\label{fruit}

Notice that the fruit graph $\Fruit_{n}$ from before is simply $\Tadpole_{1,n-1}$. In this section, we expand upon the results in \cite{group}, generalizing the fruit graphs that consist of a cycle together with an edge connecting a new vertex with an existing one into tadpole graphs.

\begin{theorem}\label{thm:stemmedcycle}
Let $X=\Spider(a,b,1,1)$ be a graph on $n=a+b+3$ vertices with $a\geq b\geq 2$. Then $\FS(X,\overline{\Tadpole_{c,n-c}})$ is connected for any $c\geq 3$.
\end{theorem}

The proof of this theorem relies heavily on induction and can be split into multiple base cases as well as inductive steps.

\begin{itemize}
\item \textit{Base Case:} All graphs of the form $\FS(\Spider(2,2,1,1),\overline{\Tadpole_{c,7-c}})$ with $3\leq c\leq 7$ are connected. This can be verified using a computer program. Past results in \cite{group} also prove connectedness for $c=6,7$. We will also need that $\FS(\Spider(a,b,1,1),\overline{\Tadpole_{c,n-c}})$ is connected for $c=0,n$ for all $n\geq7$. However, this just reduces to the cases of $\FS(\Spider(a,b,1,1),\overline{\Cycle_n})$ and $\FS(\Spider(a,b,1,1),\overline{\Path_n})$, both of which are connected by Theorem \ref{thm:group}.

\item \textit{Inductive Step:} A lemma similar to \cite[Corollary~4.3]{group}, shown below.
\end{itemize}

\begin{lemma}\label{lemma5.3}
    Let $\FS(\Spider(a,b,1,1),\overline{\Tadpole_{c-1,n-c}})$ be connected with $a,b\geq 2$ and $n>c$, and let $\FS(\Spider(a,b,1,1),\overline{\Tadpole_{c,n-c-1}})$ also be connected. Let $X$ be a graph obtained by adding an additional vertex to $\Spider(a,b,1,1)$ and a single edge that connects the new vertex in $X$ to one of the existing ones. Then $\FS(X,\Tadpole_{c,n-c})$ is also connected.
\end{lemma}

\begin{proof}

Consider $\FS(X,\overline{\Tadpole_{c,n-c}})$. If we were to remove a label $l$ in $\overline{\Tadpole_{c,n-c}}$ (namely the one occupying the new vertex of $X$), the graph $\overline{\Tadpole_{c,n-c}}$ would be reduced to a graph on $n-1$ vertices isomorphic one of the following:
\begin{itemize}
    \item $\overline{\Spider(n-c,i-1,c-i-1)}$ for $1\leq i \leq\lfloor\frac{c-1}{2}\rfloor$, if the removed vertex is $i\geq 2$ away from the triple point and on the cycle
    \item $\overline{\Path_{n-1}}$ if the removed vertex is $1$ away from the triple point and on the cycle
    \item $\overline{\Path_{n-c-i}\cup \Tadpole_{c,i-1}}$ if the removed vertex is a distance $i\geq 1$ from the triple point and on the path
    \item $\overline{\Path_{c-1}\cup \Path_{n-c}}$, if the removed vertex \textit{is} the triple point.
\end{itemize} 

Considering each one at a time, we see that all of their complements actually contain either $\Tadpole_{c,n-c-1}$ or $\Tadpole_{c-1,n-c}$.

\begin{itemize}
    \item Notice that $\Tadpole_{c-1,n-c}$ contains all spiders of the form $\Spider(n-c,i-1,c-i-1)$ with $1\leq i\leq \lfloor\frac{c-1}{2}\rfloor$ as subgraphs 
    (the former can be attained by adding an edge between the two endpoints of the latter two legs). Thus, given that $\FS(\Spider(a,b,1,1),\overline{\Tadpole_{c-1,n-c}})$ is connected, we can also conclude that $\FS(\Spider(a,b,1,1),\overline{\Spider(n-c,i-1,c-i-1)})$ is connected for the values of $i$ concerned.
    \item Additionally, we know that $\Path_{n-1}$ is a subgraph of $\Cycle_{n-1}$, and it is known from Theorem \ref{thm:group} that $\FS(\Spider(a,b,1,1),\overline{\Cycle_{n-1}})$ is connected.
    \item Furthermore, $\Path_{n-c-i}\cup \Tadpole_{c,i-1}$ is a subgraph of $\Tadpole_{c,n-c-1}$ for all values of $i$ mentioned. However, we know that $\FS(\Spider(a,b,1,1),\overline{\Tadpole_{c,n-c-1}})$ is connected.
    \item Finally, $\Path_{c-1}\cup \Path_{n-c}$ is also a subgraph of $\Cycle_{n-1}$, so we have a similar result as in the second $\overline{\Path_{n-1}}$ case.
\end{itemize}

Because all four of the resultant friends-and-strangers graphs are connected, we can perform the necessary swaps to bring one of the friends (i.e., labels not among those that are adjacent in $\Tadpole_{c,n-c}$ to the label of the new vertex in $X$) next to the new vertex in $X$. Then, we may swap the label $l$ currently occupying the new vertex in $X$ with the friend labeled $l'$ in $X$. Now consider the remaining friends-and-strangers graph $\FS(X \setminus \{n\},\overline{\Tadpole_{c,n-c}}\setminus \{l'\})$, which is connected regardless of the label $l'$ that is removed. Then, we can repeat again by performing the necessary swaps to bring friend of $l'$, say $l''$, to the vertex adjacent to the $n$ in $X$. 

Repeating this process for $l''', l'''',$ and so forth will allow us to rotate through all possible permutations of $[n+1]$. Notice that because $n\geq 4$, we can always perform the swaps in a way that avoids returning the same label back to the new vertex in $X$. Thus $\FS(X,Y)$ is connected.
\end{proof}

\begin{proof}[Proof of Theorem \ref{thm:stemmedcycle}]
We induct on both $c$ and $n-c$, the lengths of the cycle and path in $\overline{Y}$, respectively. Our base cases tell us that all possible graphs $\FS(\Spider(2,2,1,1),\overline{\Tadpole_{c,7-c}})$ for $3\leq c\leq 7$ are connected. 
With these base cases (as well as the fact that $\FS(X, \overline{\Cycle_{n}})$ is connected for all $4$-legged spiders $X$ with $n\geq6$ vertices by Theorem \ref{thm:group}), we can repeatedly use Lemma \ref{lemma5.3} to obtain the desired result.
\end{proof}

\section{Generalizations to Spiders}\label{general}



Notice that $\FS(\Spider(\lambda_1,\lambda_2,\ldots,\lambda_k),\overline{\Spider(\lambda_1,\lambda_2,\ldots,\lambda_k)})$ is not connected, as the identity permutation $\text{id}\colon V(\Spider(\lambda_1,\lambda_2,\ldots,\lambda_k))\to V(\overline{\Spider(\lambda_1,\lambda_2,\ldots,\lambda_k)})$ is not adjacent to any other vertices in the friends-and-strangers graph. Thus, if $X$ is an $n$-vertex spider such that $\FS(X,Y)$ is connected for every $n$-vertex graph $Y$ that is the complement of a spider with at most $k$ legs, then $X$ must have at least $k+1$. In Theorem~\ref{maintheorem}, we show that for each $k\geq 3$, there exists such a spider $X$ with exactly $k+1$ legs. Additionally, the following claim shows that we cannot take $X$ to be the $(k+1)$-legged spider $\Spider(2,1,\ldots,1)$.

\begin{claim}\label{claim:necessary}
Let $X=\Spider(\lambda_1,\lambda_2,\ldots,\lambda_{k+1})$ and $Y$ be a graph such that $\overline{Y}$ has maximum degree $\lambda_2+\lambda_3+\cdots+\lambda_{k+1}$. Then $\FS(X,Y)$ is not connected.
\end{claim}

The following theorem trivializes the proof of the claim. A \emph{cut vertex} is a vertex that if removed, leaves behind a disconnected graph.

\begin{theorem}\cite{main}
Let $X$ and $Y$ be graphs on $n$ vertices. Supposed $x_1\cdots x_d$  ($d\geq1$) is a path in $X$, where $x_1$ and $x_d$ are cut vertices and each of $x_2,\ldots,x_{d-1}$ has degree exactly $2$. If the minimum degree of $Y$ is at most $d$, then $\FS(X,Y)$ is disconnected.
\label{thm:disconnected}
\end{theorem}

\begin{proof}[Proof of Claim \ref{claim:necessary}] We take the path $x_1\ldots x_d$ to be the longest spider leg excluding the foot, together with the center vertex so that $d=(\lambda_1-1)+1=\lambda_1$. Let the center vertex be $x_1$ and the vertex adjacent to the foot of the longest leg be $x_d$. It is not hard to see that the conditions in Theorem~\ref{thm:disconnected} are satisfied, and that
\begin{align*}
    \text{(minimum degree of $Y$)}&\leq\text{((maximum possible degree of $Y$)}-1) - \text{(maximum degree of $\overline{Y}$)}\\
    &=(\lambda_1+\lambda_2+\cdots+\lambda_{k+1}+1-1)-(\lambda_2+\lambda_3+\cdots+\lambda_{k+1})\\
    &=\lambda_1 \leq \lambda_1,
\end{align*}
so $\FS(X,Y)$ is indeed disconnected.
\end{proof}

\begin{theorem}\label{maintheorem}
Let $X=\Spider(2,2,1\ldots,1)$ be a spider on $k+1$ legs with $n=k+4$ vertices. The following three graphs are connected:
\begin{itemize}
    \item $\FS(X,\overline{\Spider(2,2,2,1\ldots,1)})$
    \item $\FS(X,\overline{\Spider(3,2,1,\ldots,1)})$
    \item $\FS(X,\overline{\Spider(4,1,\ldots,1)})$.
\end{itemize}
Here, there are $k$ legs in each of the complement spider graphs.
\end{theorem}

\begin{proof} For sake of clarity, we assign the following names to the various vertices in $X=\Spider(2,2,1\ldots,1)$: 
\begin{itemize}
    \item the \textit{center} vertex is the vertex of degree $k+1$,
    \item any foot that is adjacent to the center is an \textit{A-type} vertex,
    \item any vertex of degree 2 is a \textit{B-type} vertex, and
    \item any foot that is not adjacent to the center is a \textit{C-type} vertex.
\end{itemize}
An example graph is displayed in Figure \ref{fig:typevertex} with appropriate color-coded vertex types.
\color{black}
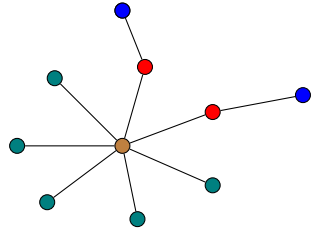
\begin{figure}[h]\label{fig:typevertex}
\begin{tikzpicture}[line width=1pt, x=0.4cm,y=0.3cm]
\tikzstyle{dot}=[circle,thin,draw,fill=black,inner sep=2pt]
\pgfdeclarelayer{nodelayer}
\pgfdeclarelayer{edgelayer}
\pgfsetlayers{edgelayer,nodelayer}
	\begin{pgfonlayer}{nodelayer}
		\node [style=dot,fill=brown] (0) at (0, 0) {};
		\node [style=dot,fill=red] (1) at (0.75, 3.5) {};
		\node [style=dot,fill=blue] (2) at (0, 6) {};
		\node [style=dot,fill=red] (3) at (3, 1.5) {};
		\node [style=dot,fill=blue] (4) at (6, 2.25) {};
		\node [style=dot,fill=teal] (5) at (3, -1.75) {};
		\node [style=dot,fill=teal] (6) at (0.5, -3.25) {};
		\node [style=dot,fill=teal] (7) at (-2.5, -2.5) {};
		\node [style=dot,fill=teal] (8) at (-3.5, 0) {};
		\node [style=dot,fill=teal] (9) at (-2.25, 3) {};
		\node [style=dot,fill=blue] (10) at (0, 6) {};
	\end{pgfonlayer}
	\begin{pgfonlayer}{edgelayer}
		\draw (10) to (1);
		\draw (1) to (0);
		\draw (0) to (3);
		\draw (3) to (4);
		\draw (0) to (5);
		\draw (0) to (6);
		\draw (0) to (7);
		\draw (8) to (0);
		\draw (0) to (9);
	\end{pgfonlayer}
\end{tikzpicture}
\caption{The graph $\Spider(2,2,1,1,1,1)$. The center vertex is \textcolor{brown}{brown}, the $A$-type vertices are \textcolor{teal}{teal}, the $B$-type vertices are \textcolor{red}{red}, and the $C$-type vertices are \textcolor{blue}{blue}.}
\color{black}
\end{figure}

Additionally, let the three complement spider graphs generally be referred to as $Y$. We focus on the center vertex of the spider $\overline{Y}$, which corresponds to the label $n$. Theoretically we can think of ``hiding" this vertex in the graph $X$ so that many of the edges in $\overline{Y}$ are effectively canceled, thus making swaps between the other $n-1$ labels much easier.
\color{black}
To achieve this, we prove the following three claims. Combining the latter two claims (and noting that all series of swaps are reversible) shows that the label $n$ can occupy any vertex in $X$, and combining this fact with the first claim proves the desired result by showing all $n!$ permutations are reachable.

\begin{claim}\label{claim1}
If the label $n$ is located at an $A$ or $C$-type vertex, the remaining $n-1$ labels can be freely swapped.
\end{claim} 

\begin{claim}\label{claim2}
It is always possible to move the label $n$ if it is located at the center or a $B$-type vertex to either an $A$ or $C$-type vertex through a series of swaps.
\end{claim}

 \begin{claim}\label{claim3}
 If the label $n$ currently occupies an $A$ or $C$-type vertex, a series of swaps can move it to any other $A$ or $C$-type vertex.
 \end{claim}

\begin{proof}[Proof of Claim \ref{claim1}] For either type of vertex, there remain exactly $3$ edges in $\overline{Y}$ after the label $n$ is removed, and in all three graphs the edges are subgraphs of $\Path_{6}$, which is itself a subgraph of $\Cycle_{n-1}$. Thus, showing all $(n-1)!$ permutations can be reached is equivalent to showing $\FS(\Spider(2,2,\triangle),\overline{\Cycle_{n-1}})$ is connected for the $A$-type vertex and $\FS(\Spider(2,1,\triangle),\overline{\Cycle_{n-1}})$ is connected for the $C$-type vertex, where the $\triangle$ represents $k-2$ 1's in both graphs. However, by Proposition \ref{prop3.5} both graphs are connected.
\end{proof}

\begin{proof}[Proof of Claim \ref{claim2}]
We begin with the case of the $B$-type vertex. If the label $n$ is not adjacent in $\overline{Y}$ to the label $l$ currently occupying the adjacent $C$-type vertex, a swap between these two labels would complete the proof. Thus, we assume otherwise: that the labels $n$ and $l$ are adjacent in $\overline{Y}$. Consider the graph $X'=\Spider(2,1,\triangle)$ obtained from $X$ by deleting the entire leg containing the label $n$. Additionally, consider the graph $Y'=Y\setminus\{l,n\}$. Note that the union of edges in $\overline{Y'}$ is a subgraph of $\Cycle_{n-2}$, so to show that the remaining $n-2$ labels in $Y'$ can be free swapped around we show $\FS(X',Y')$ is connected. Once again by Proposition \ref{prop3.5} this is true.

Thus, we can swap the labels so that the center vertex and an $A$-type vertex are both occupied by labels not adjacent to $n$ in $\overline{Y}$. From here, swap the label $n$ first with the center vertex label, then with the $A$-type vertex label so that the label $n$ now occupies an $A$-type vertex.

For the case of when the label $n$ is in the center, note that it is adjacent to exactly $n$ different labels in $\overline{Y}$, so at least one of the $n-3$ $A$ or $B$-type vertices adjacent to the center vertex will have a label that can swap with it. Once this occurs, the label $n$ occupies either an $A$ or $B$-type vertex, and for the latter we just repeat the procedure for $B$-type vertices to complete the proof.
\end{proof}

 \begin{proof}[Proof of Claim \ref{claim3}]
 We show that the label $n$ can be moved from any $A$-type vertex to any other $A$-type vertex. We also show that it can be moved from any $C$-type vertex to any $A$-type vertex. Reversing the steps of the second method shows that moving from a $A$ to any $C$-type vertex as well as going from one $C$-type vertex to the other, is possible.
 
 For the $A$ to $A$-type vertex swap, we use the fact that we can rearrange the remaining $n-1$ labels in any way we want at the start. Thus, we arrange them so that the center vertex as well as the new $A$-type vertex contain two of the three labels that are not adjacent to the label $n$ in $\overline{Y}$. Swapping the label $n$ with the center vertex label and then the new $A$-type vertex label completes the procedure. 
 
 For the $C$ to $A$-type vertex swap, we once again use the rearranging method, this time to place all three labels not adjacent to the label $n$ in $\overline{Y}$ at (1) the $B$-type vertex adjacent to the current $C$-type vertex, (2) the center vertex, and (3) the desired $A$-type vertex. Swapping the label $n$ with the labels occupying these three vertices in order completes the procedure.
 \end{proof}
 
 With these three claims, the proof is complete.
 \end{proof}
 
\color{black}
\begin{theorem}\label{thm:spiderinduct}
Let $X$ be any connected graph on $n$ vertices that contains the $(k+1)$-legged spider $\Spider(2,2,1,\ldots,1)$ as a subgraph. Let $Y$ be a graph so that $\overline{Y}$ is a spider with at most $k$ legs. Then $\FS(X,Y)$ is always connected. 
\end{theorem}

\begin{proof}
We proceed by induction on the number of vertices, the base case for different values of $k$ given by Theorem \ref{maintheorem}. Assume that for any graphs $X'$ and $Y'$ with vertex set $[n-1]$ that satisfy the constraints in the theorem statement, $\FS(X',Y')$ is connected. We can add a single edge connecting any existing foot in $X'$ to a new vertex, simultaneously adding a new edge connecting an existing foot in $\overline{Y'}$ to a new label. Denote the new graphs formed as $X$ and $Y$, respectively.

We show that regardless of the label $l$ that occupies the newly added vertex in $X$, the graph $\FS(X',\overline{Y}\setminus\{l\})$ is still connected. The graph $X'$ obviously satisfies Theorem \ref{thm:spiderinduct}, and $\overline{Y}\setminus \{l\}$ is always a subgraph of a spider with \emph{at most} $n$ legs. Thus, $\FS(X',\overline{Y}\setminus\{l\})$ is always connected.

Since this implies that 
\begin{itemize}
    \item it is always possible to perform swaps so that all permutations of the remaining labels on $X'$ can be reached, and
    \item repeated swaps between different labels occupying the new vertex in $X$ and the adjacent foot in $X'$ will guarantee that every label could be swapped into the new vertex, 
\end{itemize}
this is sufficient to prove that $\FS(X,Y)$ is connected. Since our choices of $X$ and $Y$ were arbitrary, this implies any graphs $X$ and $Y$ with $n$ vertices that satisfy the conditions of Theorem \ref{thm:spiderinduct} yield a connected $\FS(X,Y)$, completing the inductive step.

\end{proof}

\begin{ex}
The friends-and-strangers graph $\FS(\Spider(10,5,4,3,1,1),Y)$ is connected for any graph $Y$ such that $\overline{Y}$ is a spider with at most $5$ legs.
\end{ex}

Considering the overlap between sections \ref{fruit} and \ref{general}, we obtain a corollary of Theorems \ref{thm:stemmedcycle} and \ref{thm:spiderinduct}. Notice that $\Spider(a,b,c)$ is a subgraph of $\Tadpole_{a+b+1,c}$, which allows us to extend the implications of Theorem \ref{thm:stemmedcycle} to $3$-legged spiders as follows.

\begin{corollary}\label{cor:threelegs}
Let $X$ be a graph on $n$ vertices that contains $\Spider(2,2,1,1)$ as a subgraph. Let $Y=\overline{\Spider(a,b,c)}$ so that $a+b+c+1=n$. Then $\FS(X,Y)$ is connected.
\end{corollary}

\section{Concluding Remarks and Future Directions}\label{final}

A more generalized study of the tadpole graphs $\Tadpole_{c,n-c}$ in Section \ref{fruit} brings forth a new type of graph. Let $\lambda_1\geq\lambda_2\geq\cdots\geq\lambda_k$ and $\gamma_1\geq\gamma_2\geq\cdots\geq\gamma_h\geq3$ all be integers. We let the \emph{spycle graph} $\Spycle(\lambda_1,\lambda_2,\ldots,\lambda_k;\gamma_1,\gamma_2,\ldots,\gamma_h)$ be the graph on $n=1+\sum_{i=1}^k \lambda_i + \sum_{i=1}^h \gamma_i$ vertices. The vertex $n$ serves as the center vertex of a spider with $k$ legs of length $\lambda_1,\ldots,\lambda_k$, and the graph also contains $h$ cycles with length $\gamma_1,\ldots,\gamma_h$, all intersecting at only $n$. Notice the spycle graph is also a generalization of the tadpole graph $\Tadpole_{c,n-c}$, which can be expressed as $\Spycle(n-c;c)$. A picture of $\Spycle(2,4,4;3,5)$ is below.

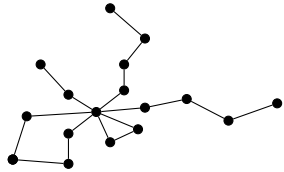
\begin{figure}[H]
\begin{tikzpicture}[line width=1pt, x=0.37cm,y=0.23cm]
\tikzstyle{dot}=[circle,thin,draw,fill=black,inner sep=1.2pt]
\pgfdeclarelayer{nodelayer}
\pgfdeclarelayer{edgelayer}
\pgfsetlayers{edgelayer,nodelayer}
	\begin{pgfonlayer}{nodelayer}
	\node [style=dot] (0) at (0, 0) {};
		\node [style=dot] (1) at (-1, 1) {};
		\node [style=dot] (2) at (-2, 2.75) {};
		\node [style=dot] (3) at (1, 1.25) {};
		\node [style=dot] (4) at (1, 2.75) {};
		\node [style=dot] (5) at (1.75, 4.25) {};
		\node [style=dot] (6) at (0.5, 6) {};
		\node [style=dot] (7) at (1.75, 0.25) {};
		\node [style=dot] (8) at (3.25, 0.75) {};
		\node [style=dot] (9) at (4.75, -0.5) {};
		\node [style=dot] (10) at (6.5, 0.5) {};
		\node [style=dot] (11) at (0.5, -1.75) {};
		\node [style=dot] (12) at (1.5, -1) {};
		\node [style=dot] (13) at (-2.5, -0.25) {};
		\node [style=dot] (14) at (-1, -1.25) {};
		\node [style=dot] (15) at (-3, -2.75) {};
		\node [style=dot] (16) at (-1, -3) {};
		\node [style=dot] (17) at (-3, -2.75) {};
		\node [style=dot] (18) at (-3, -2.75) {};
	\end{pgfonlayer}
	\begin{pgfonlayer}{edgelayer}
		\draw (13) to (18);
		\draw (18) to (16);
		\draw (16) to (14);
		\draw (14) to (0);
		\draw (0) to (1);
		\draw (1) to (2);
		\draw (0) to (3);
		\draw (3) to (4);
		\draw (4) to (5);
		\draw (5) to (6);
		\draw (0) to (7);
		\draw (7) to (8);
		\draw (8) to (9);
		\draw (9) to (10);
		\draw (0) to (12);
		\draw (12) to (11);
		\draw (11) to (0);
		\draw (0) to (13);
	\end{pgfonlayer}
\end{tikzpicture}
\end{figure}

\begin{qn}
What conditions are required for $\overline{Y}$ to satisfy in order for $\FS(\Spycle,Y)$ to be connected? If two legs of a spider $X$ are connected at their feet by an edge to form a spycle $X'$, what graphs $Y$ yield a connected friends-and-strangers graph $\FS(X',Y)$ but a disconnected $\FS(X,Y)$? 
\end{qn}

Similar to the method used to prove Theorem \ref{maintheorem}, an inductive approach is promising when investigating the connectedness of $\FS(\Spycle,\overline{\Spycle})$.

\section*{Acknowledgements}

I extend my gratitude to the MIT PRIMES organizers for arranging such an engaging research opportunity. I would also like to thank my mentor, Colin Defant, for his advice and assistance with the research and writing process. Additionally, I would like to thank Dr. Tanya Khovanova for providing insightful comments during the paper-writing process.

\end{document}